\documentclass[12pt,a4paper]{article}

\usepackage[dvips]{color}
\usepackage{color}
\usepackage{xcolor}

\usepackage{amsfonts}
\usepackage{amsfonts}
\usepackage{amsfonts}
\usepackage{amsfonts}
\usepackage{amsfonts}
\usepackage{mathrsfs}
\usepackage{amsfonts}
\usepackage{mathrsfs}
\usepackage{amsfonts}
\usepackage{amsfonts}
\usepackage{mathrsfs}
\usepackage{mathrsfs}
\usepackage{amsfonts}
\usepackage{amsfonts}
\usepackage{amsfonts}
\usepackage{amsfonts}
\usepackage{mathrsfs}

\usepackage{amsfonts}
\usepackage{mathrsfs}
\usepackage{mathrsfs}
\usepackage{mathrsfs}
\usepackage{mathrsfs}
\usepackage{amsfonts}
\usepackage{mathrsfs}
\usepackage{mathrsfs}
\usepackage{amsfonts}
\usepackage{amsfonts}
\usepackage{amsfonts}
\usepackage{amsfonts}
\usepackage{amsfonts}
\usepackage{amsfonts}
\usepackage{amsfonts}
\usepackage{amsfonts}
\usepackage{amsfonts}
\usepackage{mathrsfs}
\usepackage{amsfonts}
\usepackage{mathrsfs}
\usepackage{amsfonts}
\usepackage{mathrsfs}
\usepackage{amsfonts}
\usepackage{amsfonts}
\usepackage{amsfonts}
\usepackage{amsfonts}
\usepackage{amsfonts}
\usepackage{amsfonts}
\usepackage{amsfonts}
\usepackage{amsfonts}
\usepackage{amsfonts}
\usepackage{amsfonts}
\usepackage{amsfonts}
\usepackage{amsfonts}
\usepackage{mathrsfs}
\usepackage{mathrsfs}
\usepackage{mathrsfs}
\usepackage{mathrsfs}
\usepackage{mathrsfs}
\usepackage{mathrsfs}
\usepackage{mathrsfs}
\usepackage{mathrsfs}
\usepackage{mathrsfs}
\usepackage{amsfonts}
\usepackage{amsfonts}
\usepackage{amsfonts}
\usepackage{amsfonts}
\usepackage{amsfonts}
\usepackage{amsfonts}
\usepackage{mathrsfs}
\usepackage{amsfonts}
\usepackage{amsfonts}
\usepackage{amsfonts}
\usepackage{amsfonts}
\usepackage{amsfonts}
\usepackage{amsfonts}
\usepackage{amsfonts}
\usepackage{amsfonts}
\usepackage{amsfonts}
\usepackage{amsfonts}
\usepackage{amsfonts}
\usepackage{amsfonts}
\usepackage{mathrsfs}
\usepackage{amsfonts}
\usepackage{mathrsfs}
\usepackage{amsfonts}
\usepackage{amsfonts}
\usepackage{amsfonts}
\usepackage{mathrsfs}
\usepackage{amsfonts}
\usepackage{amsfonts}
\usepackage{amsfonts}
\usepackage{amsfonts}
\usepackage{amsmath}
\usepackage{mathrsfs}
\usepackage{amsfonts}
\usepackage{amsfonts}
\usepackage{amsfonts}
\usepackage{amsfonts}
\usepackage{amsfonts}
\usepackage{amsfonts}
\usepackage{amsfonts}
\usepackage{amsfonts}
\usepackage{amsfonts}
\usepackage{amsfonts}
\usepackage{amsfonts}
\usepackage{mathrsfs}
\usepackage{amsfonts}
\usepackage{amsfonts}
\usepackage{amsfonts}
\usepackage{amsfonts}
\usepackage{amsmath}
\usepackage{mathrsfs}
\usepackage{mathrsfs}
\usepackage{mathrsfs}
\usepackage{mathrsfs}
\usepackage{mathrsfs}
\usepackage{mathrsfs}
\usepackage{mathrsfs}
\usepackage{mathrsfs}
\usepackage{amsfonts}
\usepackage{amsfonts}
\usepackage{amsfonts}
\usepackage{amsfonts}
\usepackage{amsfonts}
\usepackage{amsfonts}
\usepackage{amsfonts}
\usepackage{amsfonts}
\usepackage{amsfonts}
\usepackage{amsfonts}
\usepackage{amsfonts}
\usepackage{amsfonts}
\usepackage{mathrsfs}
\usepackage{amsfonts}
\usepackage{mathrsfs}
\usepackage{amsfonts}
\usepackage{amsfonts}
\usepackage{amsfonts}
\usepackage{amsfonts}
\usepackage{amsfonts}
\usepackage{amsfonts}
\usepackage{amsfonts}
\usepackage{amsfonts}
\usepackage{amsfonts}
\usepackage{amsfonts}
\usepackage{amsfonts}
\usepackage{amsfonts}
\usepackage{amsfonts}
\usepackage{amsfonts}
\usepackage{amsfonts}
\usepackage{amsfonts}
\usepackage{amsfonts}
\usepackage{amsfonts}
\usepackage{mathrsfs}
\usepackage{mathrsfs}
\usepackage{mathrsfs}
\usepackage{mathrsfs}
\usepackage{amsfonts}
\usepackage{amsfonts}
\usepackage{amsfonts}
\usepackage{amsfonts}
\usepackage{amsfonts}
\usepackage{amsfonts}
\usepackage{amsfonts}
\usepackage{amsfonts}
\usepackage{amsfonts}
\usepackage{amsfonts}
\usepackage{amsfonts}
\usepackage{amsfonts}
\usepackage{amsfonts}
\usepackage{amsfonts}
\usepackage{amsfonts}
\usepackage{amsfonts}
\usepackage{amsfonts}
\usepackage{mathrsfs}
\usepackage{mathrsfs}
\usepackage{amsfonts}
\usepackage{amsfonts}
\usepackage{amsfonts}
\usepackage{amsfonts}
\usepackage{amsfonts}
\usepackage{amsfonts}
\usepackage{amsfonts}
\usepackage{amsfonts}
\usepackage{amsfonts}
\usepackage{amsfonts}
\usepackage{amsfonts}
\usepackage{amsfonts}
\usepackage{amsfonts}
\usepackage{mathrsfs}
\usepackage{mathrsfs}
\usepackage{amsfonts}
\usepackage{amsfonts}
\usepackage{amsfonts}
\usepackage{amsfonts}
\usepackage{amsfonts}
\usepackage{amsfonts}
\usepackage{mathrsfs}
\usepackage{mathrsfs}
\usepackage{amsfonts}
\usepackage{amsfonts}
\usepackage{amsfonts}
\usepackage{amsfonts}
\usepackage{amsfonts}
\usepackage{amsfonts}
\usepackage{amsfonts}
\usepackage{amsfonts}
\usepackage{amsfonts}
\usepackage{amsfonts}
\usepackage{amsfonts}
\usepackage{amsfonts}
\usepackage{amsfonts}
\usepackage{amsfonts}
\usepackage{mathrsfs}
\usepackage{amsfonts}
\usepackage{amsfonts}
\usepackage{amsfonts}
\usepackage{amsfonts}
\usepackage{amsfonts}
\usepackage{amsfonts}
\usepackage{amsfonts}
\usepackage{amsfonts}
\usepackage{amsfonts}
\usepackage{amsfonts}
\usepackage{amsfonts}
\usepackage{amsfonts}
\usepackage{amsfonts}
\usepackage{amsfonts}
\usepackage{amsfonts}
\usepackage{amsfonts}
\usepackage{amsmath}
\usepackage{mathrsfs}
\usepackage{amsfonts}
\usepackage{amsfonts}
\usepackage{amsfonts}
\usepackage{amsfonts}
\usepackage{amsmath}
\usepackage{mathrsfs}
\usepackage{amsfonts}
\usepackage{mathrsfs}
\usepackage{mathrsfs}
\usepackage{mathrsfs}
\usepackage{mathrsfs}
\usepackage{amsmath}
\usepackage{mathrsfs}
\usepackage{amsfonts}
\usepackage{color}
\usepackage{mathrsfs}

\usepackage{stmaryrd}
\usepackage{amsmath}
\usepackage{amssymb}

\usepackage{bbm}
\usepackage{mathrsfs}
\usepackage{graphicx}

\usepackage{enumerate}
\usepackage{thm}

\makeatletter
\def\tank#1{\protected@xdef\@thanks{\@thanks
        \protect\footnotetext[0]{#1}}}
\def\bigfoot{

    \@footnotetext}
\makeatother

\newcommand{\ea}{\end{array}}
\newtheorem{thm}{Theorem}[section]

\newtheorem{lem}{Lemma}[section]
\newtheorem{defi}{Definition}[section]
\newtheorem{rmk}{Remark}[section]

\newtheorem{exm}{Example}[section]
\numberwithin{equation}{section}

{\theorembodyfont{\rmfamily}

\newenvironment{proof}{Proof}{\hfill $\Box$}

\def\RR{\mathbb{R}}
\def\PP{\mathbb{P}}
\def\EE{\mathbb{E}}
\def\NN{\mathbb{N}}

\def\cB{{\mathcal B}}

\def\si{{\sigma}}

\def\et{{\eta}}

\def\si{{\sigma}}

\def\EE{\mathbb{ E}}

\def\si{{\sigma}}

\setcounter{equation}{0}

\allowdisplaybreaks

\begin{document}
\title{\Large \bf Irreducibility of stochastic complex Ginzburg-Landau equations driven by pure jump noise and its applications}
\date{}
\author{{Hao Yang}$^1$\footnote{E-mail:yanghao@hfut.edu.cn}~~~{Jian Wang}$^2$\footnote{E-mail:wg1995@mail.ustc.edu.cn}~~~ {Jianliang Zhai}$^2$\footnote{E-mail:zhaijl@ustc.edu.cn}
\\
 \small 1. School of Mathematics, Hefei University of Technology, Hefei, Anhui 230009, China. \\
  \small 2. School of Mathematical Sciences,
 \small  University of Science and Technology of China,\\
 \small  Hefei, Anhui 230026, China.\\
}
\maketitle

\begin{center}
\begin{minipage}{130mm}
{\bf Abstract:}
Considering irreducibility is fundamental for studying the ergodicity of stochastic dynamical systems. In this paper, we establish the irreducibility  of stochastic complex Ginzburg-Laudau equations driven by pure jump noise. Our results are dimension free and the conditions placed on  the driving noises are very
mild. A crucial role is played by criteria developed by the authors of this paper and T. Zhang for the irreducibility of stochastic equations driven by pure jump noise.  As  an application, we obtain the ergodicity of  stochastic complex Ginzburg-Laudau equations. We remark that our ergodicity result
covers the weakly dissipative case with pure jump degenerate noise.


\vspace{3mm} {\bf Keywords:}
 Irreducibility; Pure jump noise; Complex Ginzburg-Laudau equation;  Ergodicity.

\vspace{3mm} {\bf AMS Subject Classification (2020):}
60H15; 60G51; 37A25.
\end{minipage}
\end{center}

\newpage

\renewcommand\baselinestretch{1.2}
\setlength{\baselineskip}{0.28in}
\section{Introduction and motivation}\label{Intr}

Let $H$ be a topological space with Borel $\sigma$-field $\mathcal{B}(H)$, and let $\mathbb{X}:=\{X^x(t),t\geq0;x\in H\}$ be an $H$-valued  Markov process on some
probability space $(\Omega,\mathcal{F},\PP)$.
$\mathbb{X}$ is said to be strongly irreducible in $H$ if for each $t>0$ and $x\in H$
\begin{center}
$\PP(X^x(t)\in B)>0$ \quad for any nonempty open set $B$.
\end{center}
$\mathbb{X}$ is said to be weakly irreducible (also called accessible) to
$x_0\in H$ if the resolvent $R_{\lambda}, \lambda>0$ satisfies
\begin{equation*}
R_{\lambda}(y,U)=\lambda \int_0^{\infty}e^{-\lambda t}\PP(X^x(t)\in U)dt>0
\end{equation*}
for all $x\in H$ and all neighborhoods $U$ of $x_0$, where $\lambda > 0$ is arbitrary. It is clear that strong irreducibility implies accessibility.

Irreducibility is a fundamental property of stochastic dynamic systems.
The importance of the study of this irreducibility lies in  its relevance in the analysis of the ergodicity of  Markov processes.
The uniqueness of the invariant measures is usually obtained by proving irreducibility and the strong Feller property, or the asymptotic strong Feller property, or the $e$-property; see \cite{DZ 1996, D, DMT 1995, Hairer M 2006, Kapica, Komorowski, PZ, Z}. The main aim of the current paper is to  study the irreducibility of stochastic complex Ginzburg-Landau equations driven by pure jump noise.

The Ginzburg-Landau equation was proposed by physicists Ginzburg and Landau in the 1950s as a low-temperature superconducting model \cite{GL}. The model, for which Ginzburg and Landau won the Nobel prize in physics in 2003, is widely used in fields such as superconductivity, superfluidity, Bose-Einstein condensation, and physical phase transition processes \cite{aran}, in particular, the model is used in the description of spatial pattern formation and the onset of instabilities in nonequilibrium fluid dynamical systems \cite{cross, doering, GL1}.



We now survey previous results concerning the irreducibility of stochastic Ginzburg-Landau equations driven by pure jump noise.
To do this, we first introduce the so-called cylindrical pure jump L\'evy processes defined by the orthogonal expansion
\begin{eqnarray}\label{eq Intro 1}
L(t)=\sum_{i}\beta_iL_i(t)e_i,\ \ t\geq0,
\end{eqnarray}
where $\{e_i\}$ is an orthonormal basis of a  separable Hilbert space $H$, $\{L_i\}$ are real-valued i.i.d. pure jump L\'evy processes, and $\{\beta_i\}$ is a given sequence of nonzero real numbers.

%

 In 2013, the authors in \cite{Xulihu} obtained the accessibility to zero of  stochastic real-valued Ginzburg-Landau equations on torus $\mathbb{T}=\mathbb{R}\setminus \mathbb{Z}$ in $H:=\{h\in L^2(\mathbb{T}):\int_\mathbb{T}h(y)dy=0\}$; see the proof of \cite[Theorem 2.4]{Xulihu}.
 The driving noises they considered are the so-called cylindrical symmetric $\alpha$-stable processes with $\alpha\in(1,2)$, which have the form (\ref{eq Intro 1}) with  $\{L_i\}$ replaced by real-valued i.i.d. symmetric $\alpha$-stable processes.
 A key point in their analysis is to prove the claim that the
stochastic convolutions with respect to $L$ stay, with positive probability, in an arbitrary
small ball with zero centre on some path spaces. 
Subsequently, the authors solved a control problem to obtain the accessibility.
Some technical restrictions are placed
on the driving noises. For example,
(ii) on page 3713 of \cite{Xulihu}, i.e.,
\begin{eqnarray}\label{eq intro 2}
\alpha\in(1,2)
\end{eqnarray}
and
\begin{eqnarray}\label{eq intro 3}
C_1\gamma_i^{-\beta}\leq |\beta_i|\leq C_2\gamma_i^{-\beta}\text{ with }\beta>\frac{1}{2}+\frac{1}{2\alpha}\text{ for some positive constants }C_1\text{ and }C_2,
\end{eqnarray}
here $\{\gamma_i=4\pi^2|i|^2\}$ are the eigenvalues of the Laplace operator on $H$.
Under the same assumptions of \cite{Xulihu}, in 2017, the authors in \cite{WXX}  established  the strong irreducibility of stochastic real-valued Ginzburg-Landau equations; see \cite[Theorem 2.3]{WXX}. To do this, they established a support result for stochastic convolutions with respect to $L$  on some suitable path space; see \cite[Lemma 3.2]{WXX}. They then solved a new control problem with polynomial term to obtain the irreducibility.
By improving the methods in \cite{WXX}, in 2018, the authors in \cite{Xulihu0} obtained the strong irreducibility of stochastic real-valued Ginzburg-Landau equations driven by subordinated cylindrical Wiener process with a $\alpha/2$-stable subordinator, $\alpha\in(1,2)$; see \cite[Theorem 2.2]{Xulihu0}. Since the main ideas of \cite{Xulihu0} are similar to that of \cite{WXX}, the technical restrictions (\ref{eq intro 2}) and (\ref{eq intro 3}) on the driving noises  are
required in \cite{Xulihu0}.

We remark that all the prior results on the irreducibility of  stochastic Ginzburg-Landau equations driven by pure jump noise  concern the real-valued and one-dimensional case. The previous methods regarding the irreducibility of stochastic real-valued Ginzburg-Landau equations driven by pure jump noise  are basically
along the same lines as that of the Gaussian case; that is, two ingredients play a very important
role: the (approximate) controllability of the associated PDEs and the support of stochastic convolutions on path spaces.
These methods always  have some
very restrictive assumptions on the driving noises, such as that the driving noises are additive type and  in the class of stable processes, and  technical assumptions such as  (\ref{eq intro 2}) and (\ref{eq intro 3}) are required. The use of those methods to deal with the case of  other types of pure jump noises is unclear. Moreover, using these methods to study the irreducibility  of stochastic complex Ginzburg-Landau equations would be very hard, if not impossible.

To the best of our knowledge, there are no results on the irreducibility  of stochastic complex Ginzburg-Landau equations driven by pure  jump noise. This strongly motivates the current paper. In this paper, on the one hand, we obtain the strong irreducibility  of stochastic complex Ginzburg-Landau equations driven by multiplicative pure  jump nondegenerate noise. The conditions placed on the driving noises are very
mild, including a large class of compound Poisson processes and L\'evy processes with heavy tails such as cylindrical symmetric and non-symmetric $\alpha$-stable processes with $\alpha\in(0,2)$ and subordinated cylindrical Wiener processes with a $\alpha/2$-stable subordinator, $\alpha\in(0,2)$, etc. Therefore, our results not only cover all of the previous results but also remove certain technical restrictions required in those results, which we have mentioned above. See Theorems \ref{thmmulti-Yang} and \ref{thmmulti-Yang1} in this paper. On the other hand, we establish the accessibility  of stochastic complex Ginzburg-Landau equations driven by additive pure  jump noise and the driving noises could be degenerate. See Theorem \ref{thm2} in this paper.  To prove these main results in this paper,  a crucial role is played by the criteria for the irreducibility of stochastic equations driven by pure jump noise developed in \cite{WYZZ,WYZZ1} by the authors of this paper and T. Zhang.  As an application, we obtain the ergodicity of stochastic complex Ginzburg-Laudau equations. We remark that our ergodicity results
cover the weakly dissipative case with pure jump degenerate noise.
See Section 5 in this paper. Finally, we point out that all of
our results are dimension free.

The organization of the paper is as follows. Section 2 presents the well-posedness of stochastic complex Ginzburg-Landau equations driven by pure jump noise. In Sections 3 and 4, we prove the strong irreducibility and accessibility of stochastic complex Ginzburg-Landau equations, respectively. In Section 5, we  apply our main results to obtain the ergodicity of stochastic complex Ginzburg-Landau equations.
\section{Preliminaries and Well-posedness}
In this section, we  introduce stochastic complex Ginzburg-Landau equations driven by pure jump noise and present its well-posedness.

Let  $(\Omega,\mathcal{F},\{{\mathcal{F}}_t\}_{t\geq 0},\mathbb{P})$  be a complete probability space with a filtration $\{{\mathcal{F}}_t\}_{t\geq 0}$ satisfying the usual conditions.
Let $d\in\mathbb{N}$ and $D$ be a bounded open domain in $\mathbb{R}^d$ with smooth boundary $\partial D$.
$L^q=L^q(D)$ stands for the space of complex-valued measurable functions $u$ satisfying the Dirichlet boundary
condition, i.e., $u(x)=0, x\in \partial D$, such that
$$
\|u\|_{L^q}=\Big(\int_D|u(x)|^qdx\Big)^{1/q}<\infty.
$$
We regard $H=L^2$ as a complex Hilbert space with the scalar product
$$
\langle u,v\rangle={{\rm{Re}}}\int_Du(x)\overline{v(x)}dx,\ u,v\in H,
$$
and denote by $\|\cdot\|$ the corresponding norm.
$V=H_0^1$ is the space of functions that belong to the complex-valued Sobolev space
$H^1(D)$ and satisfy the Dirichlet boundary condition. We provide $V$ with the norm
$$
\|u\|_V=\Big(\int_D|\nabla u(x)|^2dx\Big)^{1/2},\ u\in V.
$$

Let $(Z, \cB(Z))$ be a metric space, and $\nu$ a given $\si$-finite  measure on it; that is, there exists $Z_n\in\mathcal{B}(Z), n\in\mathbb{N}$ such that $Z_n\uparrow Z$ and $\nu(Z_n)<\infty, \forall n\in\mathbb{N}$. Let $N: \cB(Z\times\RR^+) \times \Omega\rightarrow \bar{\NN} := \NN \cup \{0, \infty\}$ be a time-homogeneous Poisson random measure on $(Z, \cB(Z) )$ with intensity measure $\nu$. For the existence of such Poisson random measure, we refer the reader to \cite{IW1989}.
We denote by $\tilde{N} (dz,dt) = N(dz,dt) - \nu(dz)dt$ the compensated Poisson random measure associated to $N$.

 Consider the following stochastic Ginzburg-Landau equations with Dirichlet boundary condition driven by pure jump noise:
\begin{eqnarray}\label{1}
        du_t &=& \big((\alpha_1+i\beta_1)\Delta u_t+(\alpha_2+i\beta_2)|u_t|^{2\theta}u_t+\lambda u_t\big)dt +\int_{Z^c_1}\si(u_{t-}, z)\tilde{N}(dt, dz) \nonumber\\
        && +\int_{Z_1} \si(u_{t-}, z)N(dt, dz),\nonumber\\
        u_t&=&0 \text{ on } \partial D \text{ for any }  t\geq 0,
      \end{eqnarray}
 where $\si: H\times Z \rightarrow H$ is a  measurable mapping,
$\theta>0,~\beta_1,\beta_2,\lambda\in\mathbb{R},~\alpha_1>0$, $\alpha_2<0$ are constants, and for any $m\in\mathbb{N}$, $Z_m^c$ denotes the complement of $Z_m$ relative to $Z$.

Here is the definition of a solution to (\ref{1}).
\begin{defi}
An $H$-valued c\`{a}dl\`{a}g $\mathcal{F}_t$-adapted process $u=(u_t)_{t\in[0,\infty)}$ is called a solution of (\ref{1}), if there exists a $dt\times\mathbb{P}$-equivalent class $\hat{u}$ of $u$ such that
\par
(1) $\hat{u}\in L^2_{loc}([0,\infty);V)\cap L^{2\theta+2}_{loc}([0,\infty);L^{2\theta+2}(D)),~\mathbb{P}$-a.s.
\par
(2) The following equality holds, $\mathbb{P}$-a.s., for any $t\geq0$,
\begin{eqnarray*}
        u_t &=& u_0+\int_0^t\big((\alpha_1+i\beta_1)\Delta \hat{u}_s+(\alpha_2+i\beta_2)|\hat{u}_s|^{2\theta}\hat{u}_s+\lambda {u}_s\big)ds +\int_0^t\int_{Z_1^c}\si({u}_{s-}, z)\tilde{N}(ds, dz) \nonumber\\
        && +\int_0^t\int_{Z_1} \si({u}_{s-}, z)N(ds, dz).
      \end{eqnarray*}
The above equation is interpreted as an equation in $V^*$, the dual space of $V$.
\end{defi}

We introduce the following conditions.

{\bf (C1)}$: \alpha_2+\frac{\theta|\beta_2|}{\sqrt{2\theta+1}}< 0$.

{\bf (C2)}: There exists a positive constant $k_0$ such that
\begin{equation*}
  \int_{Z_1^c}\|\sigma(0,z)\|^2\nu(dz)\leq k_0.
\end{equation*}

{\bf (C3)}: There exist positive constants $k_1$ and $2\leq p<2\theta+2$ such that,  for all $v_1,~v_2\in H$,
\begin{equation*}
 \int_{Z_1^c} \|\sigma(v_1,z)-\sigma(v_2,z)\|^2\nu(dz)\leq k_1(\|v_1-v_2\|^2\vee\|v_1-v_2\|^p).
\end{equation*}

{\bf (C4)}: For any $z\in Z_1$, $\sigma(\cdot,z):H\rightarrow H$ is continuous.
\begin{rmk}\label{rmk-yang}
Conditions (C2) and (C3) implies that for any $v\in H$
\begin{equation}\label{2-1}
  \int_{Z_1^c}\|\sigma(v,z)\|^2\nu(dz)\leq 2(k_0+k_1+k_1\|v\|^p).
\end{equation}
\end{rmk}

Now we state the result on the existence and uniqueness of the solution to (\ref{1}).
\begin{thm}\label{thm1}
Under Conditions (C1)--(C4), for any $u_0\in H$, there exists a unique solution $u^{u_0}=(u^{u_0}_t)_{t\geq 0}$ to  (\ref{1}) with the initial data $u_0$. Moreover, $\{u^{x}\}_{x\in H}$ forms a strong Markov process.
\end{thm}
Combining the weak convergence argument and monotonicity argument as in \cite{BZ} and the idea of proving \cite[Theorem 2.5]{Lin}, under Conditions (C1)--(C3), one can establish the well-posedness of (\ref{1}). The strong Markov property of $\{u^{x}\}_{x\in H}$ follows from the Feller property, the fact $u^{x}\in D([0,\infty);H),$ $\mathbb{P}$-a.s., and Condition (C4). Since our primary concern in this paper is the irreducibility of $\{u^{x}\}_{x\in H}$, the proof of Theorem \ref{thm1} is omitted here.

\vskip 0.2cm
In the sequel, the symbol $C$ will denote a positive generic constant whose value may
change from line to line. For any $x\in H$, let $u^x=(u^x_t)_{t\geq0}$ be the unique solution to (\ref{1}) with the initial data $x$.

\section{Strong Irreducibility}
In this section, we study the strong irreducibility of (\ref{1}).

%

To study the strong irreducibility of (\ref{1}), we introduce a nondegenerate condition on the intensity measure $\nu$, which basically says that for any $\hbar,y\in H$, one can reach the neighborhoods of $y$ from $\hbar$ through a finite number of choosing jumps.

{\bf (C5)}: For any $\hbar,y\in H$ with $\hbar\neq y$ and any $\bar{\eta}>0$, there exist $n,m\in\mathbb{N}$, and $\{l_i,i=1,2,...,n\}\subset Z_m$
 such that,
 for any $\eta\in(0,\frac{\bar{\eta}}{2})$, there exist $\{\epsilon_i,i=1,2,...,n\}\subset(0,\infty)$ and $\{\eta_i,i=0,1,...,n\}\subset(0,\infty)$ such that the following hold, denoting
$$
q_0=\hbar,\ q_{i}=q_{i-1}+\sigma(q_{i-1},l_{i}),\ i=1,2,...,n,
$$
 \begin{itemize}
\item  $0<\eta_0\leq \eta_1\leq ...\leq \eta_{n-1}\leq \eta_n\leq \eta$;
   \item for any $i=0,1,...,n-1$, $\{\tilde{q}+\sigma(\tilde{q},l):\ \tilde{q}\in B(q_i,\eta_i),l\in B(l_{i+1},\epsilon_{i+1})\}\subset B(q_{i+1},\eta_{i+1})$;
   \item $B(q_{n},\eta_{n})\subset B(y,\frac{\bar{\eta}}{2})$;
   \item for any $i=1,2,...,n$, $\nu(B(l_{i},\epsilon_i))>0$;
   \item there exists $m_0\geq m$ such that $\bigcup_{i=1}^n B(l_{i},\epsilon_{i})\subset Z_{m_0}$.
 \end{itemize}
\vskip 0.3cm

We have the following main result in this paper.
\begin{thm}\label{thmmulti-Yang}
Under Conditions (C1)--(C5), the solution $\{u^x\}_{x\in H}$ of (\ref{1}) is strongly irreducible in $H$.
\end{thm}

To prove Theorem \ref{thmmulti-Yang}, the following result taken from \cite[Lemmas 2.1 and 2.2]{Okazawa} will be used.
\begin{lem}\label{lemma1}
Let $U$ be a complex Hilbert space with inner product $(\cdot,\cdot)$ and norm $\|\cdot\|_U$. Then
\begin{itemize}
  \item for any $\gamma>1$ and  any nonzero $a, b\in U$ with $a\neq b$,
\begin{equation}\label{37}
\frac{|{\rm{Im}}(\|a\|_U^{\gamma-2}a-\|b\|_U^{\gamma-2}b,a-b)|}{{\rm{Re}}(\|a\|_U^{\gamma-2}a-\|b\|_U^{\gamma-2}b,a-b)}\leq\frac{|\gamma-2|}{2\sqrt{\gamma-1}}.
\end{equation}
  \item for any $p\geq 2$ and  any $a,b\in U$,
\begin{equation}
{\rm{Re}}(\|a\|_U^{p-2}a-\|b\|_U^{p-2}b,a-b)\geq2^{2-p}\|a-b\|_U^p.
\end{equation}
\end{itemize}
\end{lem}

\begin{proof}{\bf(Proof of Theorem \ref{thmmulti-Yang})}

Applying \cite[Theorem 2.1]{WYZZ} and Theorem \ref{thm1} in this paper, we see that the proof of this theorem will be complete once
we prove the following claim.

For any $x,y\in H$, $\eta>0$, define the stopping time $\tau_{x,y}^\eta=\inf\{t\geq 0; u^x_t\not\in B(y,\eta)\}$. It is easy to see that
$\tau_{x,x}^\eta>0$, $\mathbb{P}$-a.s.

{\bf Claim 1}: For any $h\in H$, there exists $\eta_h>0$ such that, for any $\eta\in(0,\eta_h]$, there exist $(\epsilon,t)=(\epsilon(h,\eta),t(h,\eta))\in (0,\frac{\eta}{2}]\times (0,\infty)$ satisfying,
\begin{eqnarray*}
 \inf_{\tilde{h}\in B(h,\epsilon)}\mathbb{P}(\tau_{\tilde{h},h}^\eta\geq t)>0.
\end{eqnarray*}

The claim's proof is divided into two steps.

{\bf Step 1.}
Removing the large jumps from (\ref{1}), let us consider the following equation
 \begin{eqnarray}\label{acce}
 d\tilde{u}_t&=&\big((\alpha_1+i\beta_1)\Delta \tilde{u}_t+(\alpha_2+i\beta_2)|\tilde{u}_t|^{2\theta}\tilde{u}_t+\lambda \tilde{u}_t\big)dt +\int_{Z_1^c}\si(\tilde{u}_{t-}, z)\tilde{N}(dt, dz),\nonumber\\
  \tilde{u}_t&=&0 \text{ on } \partial D \text{ for any }  t\geq0.
\end{eqnarray}

Denote by $\tilde{u}_t^{h}$ and $\tilde{u}_t^{\tilde{h}}$ the solutions of (\ref{acce}) with the initial data $h,\tilde{h}\in H$, respectively.  Using the It\^{o} formula, we have
\begin{eqnarray}\label{yang-2}
&&\|\tilde{u}_t^{\tilde{h}}-\tilde{u}_t^h\|^2\nonumber\\
&=&\|\tilde{h}-h\|^2-2\alpha_1\int_0^t\|\tilde{u}_s^{\tilde{h}}-\tilde{u}_s^h\|_V^2ds+2\lambda\int_0^t\|\tilde{u}_s^{\tilde{h}}-\tilde{u}_s^h\|^2ds \nonumber\\
&&+2{\rm{Re}}\int_0^t\langle \tilde{u}_s^{\tilde{h}}-\tilde{u}_s^h,(\alpha_2+i\beta_2)(|\tilde{u}_s^{\tilde{h}}|^{2\theta}\tilde{u}_s^{\tilde{h}}-|\tilde{u}_s^h|^{2\theta}\tilde{u}_s^h)\rangle ds\nonumber\\
&&+2{\rm{Re}}\int_0^t\int_{Z_1^c}\langle \tilde{u}_{s-}^{\tilde{h}}-\tilde{u}_{s-}^h,\sigma(\tilde{u}_{s-}^{\tilde{h}},z)-\sigma(\tilde{u}_{s-}^h,z) \rangle\tilde N(ds,dz)\nonumber\\
&&+\int_0^t\int_{Z_1^c}\|\sigma(\tilde{u}_{s-}^{\tilde{h}},z)-\sigma(\tilde{u}_{s-}^h,z)\|^2N(ds,dz).
\end{eqnarray}

For the fourth term on the right side  of the above equality, we apply Lemma \ref{lemma1} with $\gamma=p=2\theta+2$  and Condition (C1) to obtain
\begin{eqnarray}\label{yang-3}
&&2{\rm{Re}}\langle \tilde{u}_s^{\tilde{h}}-\tilde{u}_s^h,(\alpha_2+i\beta_2)(|\tilde{u}_s^{\tilde{h}}|^{2\theta}\tilde{u}_s^{\tilde{h}}-|\tilde{u}_s^h|^{2\theta}\tilde{u}_s^h)\rangle\nonumber\\
&=&2\alpha_2{\rm{Re}}\langle \tilde{u}_s^{\tilde{h}}-\tilde{u}_s^h,|\tilde{u}_s^{\tilde{h}}|^{2\theta}\tilde{u}_s^{\tilde{h}}-|\tilde{u}_s^h|^{2\theta}\tilde{u}_s^h\rangle\nonumber\\
&&-2\beta_2{\rm{Im}}\langle \tilde{u}_s^{\tilde{h}}-\tilde{u}_s^h,|\tilde{u}_s^{\tilde{h}}|^{2\theta}\tilde{u}_s^{\tilde{h}}-|\tilde{u}_s^h|^{2\theta}\tilde{u}_s^h\rangle\nonumber\\
&\leq&2(\alpha_2+\frac{\theta|\beta_2|}{\sqrt{2\theta+1}}){\rm{Re}}\langle \tilde{u}_s^{\tilde{h}}-\tilde{u}_s^h,|\tilde{u}_s^{\tilde{h}}|^{2\theta}\tilde{u}_s^{\tilde{h}}-|\tilde{u}_s^h|^{2\theta}\tilde{u}_s^h\rangle\nonumber\\
&\leq&2^{1-2\theta}(\alpha_2+\frac{\theta|\beta_2|}{\sqrt{2\theta+1}})\|\tilde{u}_s^{\tilde{h}}-\tilde{u}_s^h\|^{2\theta+2}_{L^{2\theta+2}}\leq 0.
\end{eqnarray}
Since $\alpha_1>0$, (\ref{yang-2}) and (\ref{yang-3}) imply that, for any $t\geq 0$,
\begin{eqnarray}\label{yang-10}
&&\|\tilde{u}_t^{\tilde{h}}-\tilde{u}_t^h\|^2\nonumber\\
&\leq&\|\tilde{h}-h\|^2+2\lambda\int_0^t\|\tilde{u}_s^{\tilde{h}}-\tilde{u}_s^h\|^2ds\nonumber\\
&&+2^{1-2\theta}(\alpha_2+\frac{\theta|\beta_2|}{\sqrt{2\theta+1}})\int_0^t\|\tilde{u}_s^{\tilde{h}}-\tilde{u}_s^h\|^{2\theta+2}_{L^{2\theta+2}}ds\nonumber\\
&&+2{\rm{Re}}\int_0^t\int_{Z_1^c}\langle \tilde{u}_{s-}^{\tilde{h}}-\tilde{u}_{s-}^h,\sigma(\tilde{u}_{s-}^{\tilde{h}},z)-\sigma(\tilde{u}_{s-}^h,z)\rangle \tilde N(ds,dz)\nonumber\\
&&+\int_0^t\int_{Z_1^c}\|\sigma(\tilde{u}_{s-}^{\tilde{h}},z)-\sigma(\tilde{u}_{s-}^h,z)\|^2\tilde N(ds,dz)\nonumber\\
&&+
\int_0^t\int_{Z_1^c}\|\sigma(\tilde{u}_{s}^{\tilde{h}},z)-\sigma(\tilde{u}_{s}^h,z)\|^2\nu(dz)ds\nonumber\\
&\leq&\|\tilde{h}-h\|^2+2\lambda\int_0^t\|\tilde{u}_s^{\tilde{h}}-\tilde{u}_s^h\|^2ds\nonumber\\
&&+2^{1-2\theta}(\alpha_2+\frac{\theta|\beta_2|}{\sqrt{2\theta+1}})\int_0^t\|\tilde{u}_s^{\tilde{h}}-\tilde{u}_s^h\|^{2\theta+2}_{L^{2\theta+2}}ds\nonumber\\
&&+2{\rm{Re}}\int_0^t\int_{Z_1^c}\langle \tilde{u}_{s-}^{\tilde{h}}-\tilde{u}_{s-}^h,\sigma(\tilde{u}_{s-}^{\tilde{h}},z)-\sigma(\tilde{u}_{s-}^h,z)\rangle \tilde N(ds,dz)\nonumber\\
&&+\int_0^t\int_{Z_1^c}\|\sigma(\tilde{u}_{s-}^{\tilde{h}},z)-\sigma(\tilde{u}_{s-}^h,z)\|^2\tilde N(ds,dz)\nonumber\\
&&+
k_1\int_0^t\big[\|\tilde{u}_s^{\tilde{h}}-\tilde{u}_s^h\|^2\vee\|\tilde{u}_s^{\tilde{h}}-\tilde{u}_s^h\|^p\big]ds\nonumber\\
&\leq&\|\tilde{h}-h\|^2+Ct
+2{\rm{Re}}\int_0^t\int_{Z_1^c}\langle \tilde{u}_{s-}^{\tilde{h}}-\tilde{u}_{s-}^h,\sigma(\tilde{u}_{s-}^{\tilde{h}},z)-\sigma(\tilde{u}_{s-}^h,z)\rangle \tilde N(ds,dz)\nonumber\\
&&+\int_0^t\int_{Z_1^c}\|\sigma(\tilde{u}_{s-}^{\tilde{h}},z)-\sigma(\tilde{u}_{s-}^h,z)\|^2\tilde N(ds,dz).
\end{eqnarray}
To get the second and last inequalities of (\ref{yang-10}), Conditions (C1) and (C3) and the following fact have been used. Condition (C1) and $2\leq p<2\theta+2$ imply that there exists a constant $C$ such that
\begin{eqnarray*}
&&2\lambda\|\tilde{u}_s^{\tilde{h}}-\tilde{u}_s^h\|^2
+2^{1-2\theta}(\alpha_2+\frac{\theta|\beta_2|}{\sqrt{2\theta+1}})\|\tilde{u}_s^{\tilde{h}}-\tilde{u}_s^h\|^{2\theta+2}_{L^{2\theta+2}}\\
&&+
k_1\big[\|\tilde{u}_s^{\tilde{h}}-\tilde{u}_s^h\|^2\vee\|\tilde{u}_s^{\tilde{h}}-\tilde{u}_s^h\|^p\big]
\leq C<\infty.
\end{eqnarray*}

Applying stochastic Gronwall's inequality (see \cite[Lemma 3.7]{XZ}), we deduce from (\ref{yang-10}) that for any $0<q<p<1$ and any $T>0$,
\begin{eqnarray}\label{eq mono pro}
\EE\Big[\Big(\sup_{0\leq t \leq T}\|\tilde{u}_t^{\tilde{h}}-\tilde{u}_t^h\|^2\Big)^q\Big]
\leq
(\frac{p}{p-q})\Big(\|\tilde{h}-h\|^{2}+CT\Big)^q.
\end{eqnarray}

Therefore, applying Chebysev's inequality, for any $\eta>0$, there exists $\tilde{\epsilon},\tilde{t}>0$ such that, for any $\tilde{h}\in B(h,\tilde{\epsilon})$,
\begin{eqnarray}\label{3-6}
\mathbb{P}\Big(\sup_{0\leq s\leq \tilde t}\|\tilde{u}_s^{\tilde{h}}-\tilde{u}_s^h\|\leq\frac{\eta}{4}\Big)
=1-\mathbb{P}\Big(\sup_{0\leq s\leq \tilde t}\|\tilde{u}_s^{\tilde{h}}-\tilde{u}_s^h\|>\frac{\eta}{4}\Big)
\geq\frac{1}{2}.
\end{eqnarray}

{\bf Step 2.}
Denote
\begin{equation*}
 N(Z_1,t):=\int_{0}^{t}\int_{Z_1}N(dz,ds),\ t\geq0.
\end{equation*}
Let $\tau_1$ be the first jumping time of the Poisson process $ N(Z_1,t),~t\geq0$, i.e.,
\begin{eqnarray}\label{stop time 1}
\tau_1=\inf\{t\geq 0; N(Z_1, t)=1\}.
\end{eqnarray}
$\tau_1$ has an exponential distribution with parameter $\nu(Z_1)<\infty$, that is,
\begin{eqnarray*}
\mathbb{P}(\tau_1> s)=e^{-\nu(Z_1)s},\ \ \ \ \mathbb{P}(\tau_1\leq s)=1-e^{-\nu(Z_1)s}.
\end{eqnarray*}
We can choose $t_1$ small enough such that
\begin{equation}\label{3-7}
0<t_1\leq \tilde{t} \text{ and }\mathbb{P}(\tau_1> t_1)>\frac{1}{2}.
\end{equation}

\vskip 0.3cm
 It is easy to see that, for any $x\in H$, $\{u^x_t,t\in[0,\tau_1)\}$ is the unique solution to (\ref{acce}) with the initial data $x$  on $t\in[0,\tau_1)$. Then,
by (\ref{3-6}) and (\ref{3-7}), for any $\tilde{h}\in B(h,\tilde{\epsilon})$,
\begin{eqnarray}\label{yang-17}
&&\mathbb{P}\big(\sup_{0\leq s\leq  t_1}\|u_s^{\tilde{h}}-u_s^h\|\leq\frac{\eta}{4}\big)\nonumber\\
&=&\mathbb{P}\big(\sup_{0\leq s\leq  t_1}\|u_s^{\tilde{h}}-u_s^h\|\leq\frac{\eta}{4},\tau_1> t_1\big)+\mathbb{P}\big(\sup_{0\leq s\leq t_1}\|u_s^{\tilde{h}}-u_s^h\|\leq\frac{\eta}{4},\tau_1\leq  t_1\big)\nonumber\\
&\geq&\mathbb{P}\big(\sup_{0\leq s\leq  t_1}\|u_s^{\tilde{h}}-u_s^h\|\leq\frac{\eta}{4},\tau_1> t_1\big)\nonumber\\
&=&\mathbb{P}\big(\sup_{0\leq s\leq  t_1}\|\tilde{u}_s^{\tilde{h}}-\tilde{u}_s^{h}\|\leq\frac{\eta}{4},\tau_1> t_1\big)\nonumber\\
&=&\mathbb{P}\big(\sup_{0\leq s\leq  t_1}\|\tilde{u}_s^{\tilde{h}}-\tilde{u}_s^{h}\|\leq\frac{\eta}{4}\big)\times\mathbb{P}\big(\tau_1> t_1\big)\nonumber\\
&\geq&\frac{1}{4}.
\end{eqnarray}
For the last equality of (\ref{yang-17}), we have used the fact that, for any $x\in H$, $\sigma\{\tilde{u}_t^x,t\geq 0\}$ and $\sigma\{\tau_1\}$ are independent.

Since $u^h\in D([0,\infty);H)$ $\mathbb{P}$-a.s., we know that there exists $t_2>0$ small enough such that
\begin{equation}\label{yang-16}
\mathbb{P}(\sup_{0\leq s\leq t_2}\|u_{s}^{h}-h\|>\frac{\eta}{4})<\frac{1}{8}.
\end{equation}
Set
\begin{equation*}
\epsilon=\frac{\eta}{8}\wedge \tilde{\epsilon}, \quad t=t_1\wedge t_2.
\end{equation*}
By (\ref{yang-17}) and (\ref{yang-16}),
\begin{eqnarray}\label{3-4}
 &&\inf_{\tilde{h}\in B(h,\epsilon)}\mathbb{P}(\tau_{\tilde{h},h}^\eta\geq  t )\nonumber\\
 &\geq& \inf_{\tilde{h}\in B(h,\epsilon)}\mathbb{P}\big(\sup_{0\leq s\leq  t}\|u_s^{\tilde{h}}-h\|\leq\frac{\eta}{2}\big)\nonumber\\
  &\geq& \inf_{\tilde{h}\in B(h,\epsilon)}\mathbb{P}\big(\{\sup_{0\leq s\leq  t}\|u_s^{\tilde{h}}-u_s^{h}\|\leq\frac{\eta}{4}\}\cap\{\sup_{0\leq s\leq t}\|u_s^{h}-h\|\leq\frac{\eta}{4}\}\big)\nonumber\\
   &\geq& \inf_{\tilde{h}\in B(h,\epsilon)}\big(1-\mathbb{P}\big(\sup_{0\leq s\leq  t}\|u_s^{\tilde{h}}-u_s^{h}\|>\frac{\eta}{4}\big)-\mathbb{P}\big(\sup_{0\leq s\leq  t}\|u_s^{h}-h\|>\frac{\eta}{4}\big)\big)\nonumber\\
   &=&\inf_{\tilde{h}\in B(h,\epsilon)}\mathbb{P}\big(\sup_{0\leq s\leq  t}\|u_s^{\tilde{h}}-u_s^{h}\|\leq\frac{\eta}{4}\big)-\mathbb{P}\big(\sup_{0\leq s\leq t}\|u_s^{h}-h\|>\frac{\eta}{4}\big)\nonumber\\
   &\geq&\inf_{\tilde{h}\in B(h,\epsilon)}\mathbb{P}\big(\sup_{0\leq s\leq  t}\|u_s^{\tilde{h}}-u_s^{h}\|\leq\frac{\eta}{4}\big)-\frac{1}{8}\nonumber\\
   &\geq&\frac{1}{8}.
\end{eqnarray}

This completes the proof of Claim 1, so the proof of Theorem \ref{thmmulti-Yang} is also complete.
\end{proof}
\vskip 0.3cm

Let us now consider the particular case of additive noise. Let $Z=H$, and let $\nu$ denote a  given  $\si$-finite intensity measure of a L\'evy process  on $H$. Recall that $\nu(\{0\})=0$ and $\int_{H}(\|z\|_H^2\wedge1)\nu(dz)<\infty$.
Let $N: \cB(H\times\RR^+) \times \Omega\rightarrow \bar{\NN}$ be the time homogeneous Poisson random measure with intensity measure $\nu$. Again $\tilde{N} (dz,dt) = N(dz,dt) - \nu(dz)dt$ denotes the  compensated Poisson random measure associated to $N$.
\vskip 0.2cm
Let us point out that (as shown by e.g., \cite[Theorems 4.23 and 6.8]{PZ 2007}) in this case  $$L(t)=\int_0^t\int_{0<\|z\|_H\leq1}z\tilde{N}(dz,ds) + \int_0^t\int_{\|z\|_H>1} zN(dz,ds), t\geq0$$ defines an $H$-valued L\'evy process.

Consider (\ref{1}) driven by $L(t)$, $t\geq0$; that is, $\sigma(\cdot,z)\equiv z$ and
\begin{eqnarray}\label{meq1-additive case}
  &&du_t= \big((\alpha_1+i\beta_1)\Delta u_t+(\alpha_2+i\beta_2)|u_t|^{2\theta}u_t+\lambda u_t\big)dt +dL(t),\nonumber\\
  &&u_t=0, \text{ on } \partial D \text{ for any } t\geq 0.
\end{eqnarray}
In this setting,  we introduce the following hypothesis regarding the intensity measure $\nu$:

{\bf (C5')}: For any $h \in H$ and $\et >0$, there exist $n\in\mathbb{N}$, a sequence of strict positive numbers $\et_1,\ \et_2, \cdots, \et_n$, and $a_1,\ a_2, \cdots,\ a_n \in H\setminus\{0\}$, such that $ 0 \notin \overline{B(a_i, \et_i)}$, $\nu\big(B(a_i, \et_i)\big)>0$, $1 \leq i \leq n$ and $\sum_{i=1}^n{ B(a_i, \et_i)}:=\{\sum_{i=1}^n h_i:h_i\in B(a_i, \et_i),  1 \leq i \leq n\} \subset B(h, \et)$.

\vskip 0.2cm
As an application of Theorem \ref{thmmulti-Yang} (see also \cite[Theorem 2.2]{WYZZ}), we have
\begin{thm}\label{thmmulti-Yang1}
Under Conditions (C1) and (C5'), the solution $\{u^x\}_{x\in H}$ to (\ref{meq1-additive case}) is strongly irreducible in $H$.
\end{thm}

For any measure $\rho$ on $H$, its support $S_\rho=S(\rho)$ is defined to be the set of $x\in H$ such that $\rho(G)>0$ for any open set $G$ containing $x$.
Set
\begin{equation}\label{a-1}
H_0:=\Big\{\sum_{i=1}^n m_ia_i,\ n,m_1,...,m_n\in\mathbb{N},\ a_i\in S_\nu\Big\}.
\end{equation}
Finally, we point out the fact based on \cite[Subsection 4.1]{WYZZ} that Condition (C5') holds if and only if
$H_0$ is dense in $H$.

\vskip 0.3cm
\begin{rmk}
 The conditions placed on  the driving noises, i.e., conditions (C5) and (C5'), are very mild, and many  examples satisfying these conditions  can be found in \cite[Subsections 4.1 and 4.2]{WYZZ}. In particular, for the case of additive noise, the driving noises include a large class of  L\'evy processes with heavy tails, i.e., the intensity measures $\nu$ satisfy for some $\alpha\in(0,2]$, $\int_{\|z\|_H>1}\|z\|_H^\alpha\nu(dz)=\infty$, e.g., the cylindrical symmetric and nonsymmetric $\alpha$-stable processes with $\alpha\in(0,2)$, and  subordinated cylindrical Wiener processes with a $\alpha/2$-stable subordinator, $\alpha\in(0,2)$, etc. The driving noises even include processes  whose intensity measures $\nu$ satisfy for any small $\alpha>0$, $\int_{\|z\|_H>1}\|z\|_H^\alpha\nu(dz)=\infty$. Also included is a large class of compound Poisson processes, which is somewhat surprising.  For more details, see \cite[Section 4]{WYZZ}.
\end{rmk}

\section{Accessibility}
In this section, we  study the accessibility of (\ref{meq1-additive case}). The driving noises could be degenerate.


We need to consider
\begin{eqnarray}\label{yang-5}
  &&dU^{\epsilon}_t= \big((\alpha_1+i\beta_1)\Delta U^{\epsilon}_t+(\alpha_2+i\beta_2)|U^{\epsilon}_t|^{2\theta}U^{\epsilon}_t+\lambda U^{\epsilon}_t\big)dt+\int_{0<\|z\|\leq \epsilon}z \tilde{N}(dt,dz),~~\epsilon\in(0,1)\nonumber\\
  &&U^{\epsilon}_t=0 \text{ on } \partial D, \forall  t\geq 0,
\end{eqnarray}
and
\begin{eqnarray}\label{yang-4}
  &&dU_t= \big((\alpha_1+i\beta_1)\Delta U_t+(\alpha_2+i\beta_2)|U_t|^{2\theta}U_t+\lambda U_t\big)dt,\nonumber\\
  &&U_t=0 \text{ on }\partial D, \forall t\geq 0.
\end{eqnarray}

By Theorem \ref{thm1}, for any $x\in H$, there exist unique global solutions  $u^x$, $U^{\epsilon,x}$ and $U^x$ to (\ref{meq1-additive case}), (\ref{yang-5}) and (\ref{yang-4}) starting from $x$, respectively.
The following condition is required to study the accessibility.

{\bf (C6)}: $\lambda\leq \alpha_1\lambda_1$. Here and in the following, $\lambda_1$ is the first eigenvalue of $-\Delta$.
\begin{thm}\label{thm2}
Under Conditions (C1) and (C6), assume that $\nu$ is symmetric,  the solution $\{u^x\}_{x\in H}$ of (\ref{meq1-additive case}) is accessible to zero.
\end{thm}

%
%
%
%
%

\begin{proof}  According to \cite[Theorem 2.1]{WYZZ1}, to complete the proof of this theorem, we only need to verify the following claims:
\begin{itemize}

 \item[{\bf (A1)}] For any $x \in H$, $\lim_{t \rightarrow \infty} \| U_t^x \|= 0$.

 \item[{\bf (A2)}] For any $t>0$ and $x \in H$, $\lim_{\epsilon \rightarrow 0} \|U^{\epsilon,x}_t - U^{x}_t \|= 0$ in probability.

  \item[{\bf (A3)}] For any $\eta>0$, there exist $(\zeta,t)=(\zeta(\eta),t(\eta))\in (0,\frac{\eta}{2}]\times (0,\infty)$ such that\\
$ \inf_{y \in B(0,\zeta)}\mathbb{P}(\sup_{s\in[0,t]}\|u^y_s\|\leq \eta)>0$.
Here $B(0,\zeta)=\{h\in H:\|h\|<\zeta\}$.
\end{itemize}
Now, we verify Claim \textbf{(A1)}.

Applying the chain rule gives
\begin{eqnarray*}
  &&\|U_t^x\|^2= \|x\|^2-2\alpha_1\int_0^t\|U_s^x\|_V^2ds+2\lambda\int_0^t\|U_s^x\|^2ds+2\alpha_2\int_0^t\|U_s^x\|^{2\theta+2}_{L^{2\theta+2}}ds.
\end{eqnarray*}
By (C6), $\alpha_2<0$, the Poinc\'are inequality and  H\"{o}lder inequality,
\begin{eqnarray}\label{5-4}
 \frac{d\|U_t^x\|^2}{dt}&=&-2\alpha_1\|U_t^x\|_V^2+2\lambda\|U_t^x\|^2+2\alpha_2\|U_t^x\|^{2\theta+2}_{L^{2\theta+2}}\nonumber\\
  &&\leq 2\alpha_2\|U_t^x\|^{2\theta+2}_{L^{2\theta+2}}\leq\frac{2\alpha_2}{m(D)^{\theta}}\|U_t^x\|^{2\theta+2},
\end{eqnarray}
where $m(D)$ is the Lebesgue measure of $D$. Applying $\alpha_2<0$ again,  (\ref{5-4}) implies that the map $t\rightarrow\|U_t^x\|^2$ is decreasing. Then by a  contradiction argument, it is easy to see that
\begin{equation}\label{4-1}
\lim_{t\rightarrow +\infty} \|U_t^x\|^2=0.
\end{equation}
Using an argument similar to that used in the proof of Theorem \ref{thmmulti-Yang},
for any $T>0$,
\begin{equation}\label{4-2}
\lim_{\epsilon\rightarrow 0}\mathbb{E}\|U^{\epsilon,x}_T-U^x_T\|^2=0,
\end{equation}
which implies Claim \textbf{(A2)}.
Claim \textbf{(A3)} follows from \textbf{Claim 1} in the proof of Theorem \ref{thmmulti-Yang}.

The proof of Theorem \ref{thm2} is complete.
\end{proof}
\section{Applications}
In this section, we apply the irreducibility obtained above to study stochastic complex Ginzburg-Landau equations driven by pure jump noise.

%

We first investigate the existence of the invariant measure.

\subsection{Existence of the invariant measure}
To prove the existence of the invariant measure, we need the following condition on $\sigma$:

{\bf (C7)}: Assume that there exists $\hat{\theta}\in(0,1]$, $r\in[0,2\theta+\hat{\theta})$ and positive constants $k_4,~k_5$ such that  for all $u\in H$
\begin{equation*}
  \int_{Z_1}\|\sigma(u,z)\|^{\hat{\theta}}\nu(dz)+\int_{Z_1^c}\|\sigma(u,z)\|^2\nu(dz)\leq k_4\|u\|^r+k_5.
\end{equation*}

\begin{thm}\label{thmmulti-02}
Under Conditions (C1)-C4) and (C7), there exists at least one invariant measure for (\ref{1}).
\end{thm}
\begin{proof}
The proof is in the spirit of \cite{Dong}. Define a function $f$ on $H$ by
\begin{equation*}
f(u)=(\|u\|^2+1)^{\frac{\hat\theta}{2}}.
\end{equation*}
By simple calculations,  for any $u,v\in H$
\begin{equation}\label{3-8}
|f(u)-f(v)|\leq|(\|u\|^2+1)^{\frac{1}{2}}-(\|v\|^2+1)^{\frac{1}{2}}|^{\hat\theta}\leq\|u-v\|^{\hat\theta},
\end{equation}
and
\begin{equation}\label{eq zhai 05}
\nabla f(u)=\frac{\hat\theta u}{(\|u\|^2+1)^{1-\frac{\hat\theta}{2}}},\quad \nabla^2 f(u)=\frac{\hat\theta\Sigma_{i=1}^{\infty}e_i\otimes e_i}{(\|u\|^2+1)^{1-\frac{\hat\theta}{2}}}-\frac{\hat\theta(2-\hat\theta)u\otimes u}{(\|u\|^2+1)^{1-\frac{\hat\theta}{2}}}.
\end{equation}
Here $\{e_i,i\in \mathbb{N}\}$ is an orthonormal basis of $H$.

Fix $x\in H$. Let $u^x$ be the unique solution to (\ref{1}) with initial data $x$. Applying the It\^{o} formula gives
\begin{eqnarray}\label{3-9}
f(u_t^{x})
&=&f(x)-\int_0^t\frac{\alpha_1\hat\theta\|u_s^{x}\|_V^2}{(\|u_s^{x}\|^2+1)^{1-\frac{\hat\theta}{2}}}ds\nonumber\\
&&+\int_0^t\frac{\alpha_2\hat\theta\|u_s^{x}\|_{L^{2\theta+2}}^{2\theta+2}}{(\|u_s^{x}\|^2+1)^{1-\frac{\hat\theta}{2}}}ds+\int_0^t\frac{\hat\theta\lambda\|u_s^{x}\|^2}{(\|u_s^{x}\|^2+1)^{1-\frac{\hat\theta}{2}}}ds\nonumber\\
&&+\int_0^t\int_{Z_1^c}[\|u_{s-}^{x}+\sigma(u_{s-}^{x},z)\|^2+1]^{\frac{\hat\theta}{2}}-[\|u_{s-}^{x}\|^2+1]^\frac{\hat\theta}{2}\tilde N(ds,dz)\nonumber\\
&&+\int_0^t\int_{Z_1}[\|u_{s-}^{x}+\sigma(u_{s-}^{x},z)\|^2+1]^{\frac{\hat\theta}{2}} -[\|u_{s-}^{x}\|^2+1]^{\frac{\hat\theta}{2}}N(ds,dz)\\
&&+\int_0^t\int_{Z_1^c}[\|u_s^{x}+\sigma(u_s^{x},z)\|^2+1]^{\frac{\hat\theta}{2}} -[\|u_s^{x}\|^2+1]^{\frac{\hat\theta}{2}}\nonumber
-{\rm{Re}}\frac{\hat\theta\langle u_s^{x},\sigma(u_s^{x},z)\rangle}{(\|u_s^{x}\|^2+1)^{1-\frac{\hat\theta}{2}}}\nu(dz)ds.
\end{eqnarray}
Using a standard stopping time argument, by Condition (C7), (\ref{3-8}), (\ref{eq zhai 05}), and  Taylor's expansion,
\begin{eqnarray}\label{yang-14}
&&\mathbb{E}f(u_t^{x})+\mathbb{E}\int_0^t\frac{\alpha_1\hat\theta\|u_s^{x}\|_V^2}{(\|u_s^{x}\|^2+1)^{1-\frac{\hat\theta}{2}}}ds\nonumber\\
&\leq&\!\!\!f(x)
+\mathbb{E}\int_0^t\frac{\alpha_2\hat\theta\|u_s^{x}\|_{L^{2\theta+2}}^{2\theta+2}}{(\|u_s^{x}\|^2+1)^{1-\frac{\hat\theta}{2}}}ds+\mathbb{E}\int_0^t\frac{\hat\theta\lambda\|u_s^{x}\|^2}{(\|u_s^{x}\|^2+1)^{1-\frac{\hat\theta}{2}}}ds\nonumber\\
&&+\mathbb{E}\int_0^t\int_{Z_1}[\|u_{s-}^{x}+\sigma(u_{s-}^{x},z)\|^2+1]^{\frac{\hat\theta}{2}} -[\|u_{s-}^{x}\|^2+1]^{\frac{\hat\theta}{2}}N(ds,dz)\nonumber\\
&&+\mathbb{E}\int_0^t\int_{Z_1^c}[\|u_s^{x}+\sigma(u_s^{x},z)\|^2+1]^{\frac{\hat\theta}{2}} -[\|u_s^{x}\|^2+1]^{\frac{\hat\theta}{2}}\nonumber
-{\rm{Re}}\frac{\hat\theta\langle u_s^{x},\sigma(u_s^{x},z)\rangle}{(\|u_s^{x}\|^2+1)^{1-\frac{\hat\theta}{2}}}\nu(dz)ds\nonumber\\
\!\!\!
&\leq&\!\!\!f(x)
+
\mathbb{E}\int_0^t\frac{\alpha_2\hat\theta\|u_s^{x}\|_{L^{2\theta+2}}^{2\theta+2}+\hat\theta\lambda\|u_s^{x}\|^2}{(\|u_s^{x}\|^2+1)^{1-\frac{\hat\theta}{2}}}ds\nonumber\\
&&+\mathbb{E}\int_0^t\int_{Z_1}\|\sigma(u_s^{x},z)\|^{\hat\theta}\nu(dz)ds+\mathbb{E}\int_0^t\int_{Z_1^c}\|\sigma(u_s^{x},z)\|^2\nu(dz)ds\nonumber\\
\!\!\!
&\leq&\!\!\!f(x)
+
\mathbb{E}\int_0^t\frac{\alpha_2\hat\theta\|u_s^{x}\|_{L^{2\theta+2}}^{2\theta+2}+\hat\theta\lambda\|u_s^{x}\|^2}{(\|u_s^{x}\|^2+1)^{1-\frac{\hat\theta}{2}}}ds
+
\mathbb{E}\int_0^t(k_4\|u_s^{x}\|^r+k_5)ds\\
\!\!\!
&\leq&\!\!\!f(x)
+
\mathbb{E}\int_0^t
   \frac{\alpha_2\hat\theta\|u_s^{x}\|_{L^{2\theta+2}}^{2\theta+2}+\hat\theta\lambda\|u_s^{x}\|^2+k_6\|u_s^{x}\|^{r+2-\hat\theta}+k_7}
         {(\|u_s^{x}\|^2+1)^{1-\frac{\hat\theta}{2}}}ds.\nonumber
\end{eqnarray}
Since $\alpha_2<0$ and $(r+2-\hat\theta)<2\theta+2$, there exists a constant $C$ such that
\begin{eqnarray*}
\alpha_2\hat\theta\|u_s^{x}\|_{L^{2\theta+2}}^{2\theta+2}+\hat\theta\lambda\|u_s^{x}\|^2
+ k_6\|u_s^{x}\|^{r+2-\hat\theta}+k_7\leq C.
\end{eqnarray*}
Hence (\ref{yang-14}) implies that
\begin{equation*}
\mathbb{E}f(u_t^{x})+\mathbb{E}\int_0^t\frac{\alpha_1\hat\theta\|u_s^{x}\|_V^2}{(\|u_s^{x}\|^2+1)^{1-\frac{\hat\theta}{2}}}ds\leq f(x)+Ct.
\end{equation*}
Therefore,
\begin{equation*}
\mathbb{E}\int_0^t\|u_s^{x}\|_V^{\hat\theta}ds
\leq
\mathbb{E}\int_0^t\frac{\|u_s^{x}\|_V^{\hat\theta}(\|u_s^{x}\|^{2-\hat\theta}+1)}{(\|u_s^{x}\|^2+1)^{1-\frac{\hat\theta}{2}}}ds
\leq
C\mathbb{E}\int_0^t\frac{\|u_s^{x}\|_V^2+1}{(\|u_s^{x}\|^2+1)^{1-\frac{\hat\theta}{2}}}ds
\leq C(1+f(x)+t).
\end{equation*}
By the classical Bogoliubov-Krylov argument(cf.\cite{Da1}), the above inequality implies  the existence of the invariant measure.
\end{proof}

\subsection{Ergodicity: The case of additive noise}
In this subsection, we will consider the uniqueness of the invariant measure of (\ref{meq1-additive case}) with the so-called weakly dissipative condition, i.e., $\lambda=\alpha_1\lambda_1$. The driving noises could be degenerate. The following theorem is the main result of this subsection.
%
%

\begin{thm}\label{thm3}
Assume that Condition (C1) and $\lambda=\alpha_1\lambda_1$ hold.  If one of the following conditions is satisfied
\begin{itemize}
  \item Condition (C5') holds;
  \item $\nu$ is symmetric;
\end{itemize}
then there exists at most one invariant measure for (\ref{meq1-additive case}). If, moreover,
\begin{eqnarray}\label{5-1}
\exists~\hat{\theta}\in(0,1],\quad \int_{Z_1}\|z\|^{\hat{\theta}}\nu(dz)<\infty,
\end{eqnarray}
then there exists a unique invariant measure for (\ref{meq1-additive case}).
\end{thm}
\begin{proof}
For any $x,y\in H$, denote by $u^x,u^y$ the unique solutions to (\ref{meq1-additive case}) with initial data $x,y$, respectively.
Applying the It\^{o} formula gives
\begin{eqnarray}\label{5-2}
\|u_t^{x}-u_t^{y}\|^2
&=&\|x-y\|^2+2\lambda\int_{0}^{t}\|u_s^{x}-u_s^{y}\|^2ds-2\alpha_1\int_0^t\|u_s^{x}-u_s^{y}\|_V^2ds\nonumber\\
&&+2{\rm{Re}}\int_0^t\langle u_s^{x}-u_s^{y},(\alpha_2+i\beta_2)(|u_s^{x}|^{2\theta}u_s^{x}-|u_s^{y}|^{2\theta}u_s^{y})\rangle ds\nonumber\\
&\leq&\|x-y\|^2.
\end{eqnarray}
To obtain the last inequality of (\ref{5-2}), we have used $\|u_s^{x}-u_s^{y}\|_V^2\geq \lambda_1 \|u_s^{x}-u_s^{y}\|^2$, $\lambda=\alpha_1\lambda_1$ and the following estimate
\begin{eqnarray}\label{5-3}
2{\rm{Re}}\int_0^t\langle u_s^{x}-u_s^{y},(\alpha_2+i\beta_2)(|u_s^{x}|^{2\theta}u_s^{x}-|u_s^{y}|^{2\theta}u_s^{y})\rangle ds\leq 0.
\end{eqnarray}
The proof of (\ref{5-3}) follows that of (\ref{yang-3}).

(\ref{5-2}) implies that $\{u^x\}_{x\in H}$  satisfies the so-called $e$-property, i.e., for any $x\in H$ and Lipschitz bounded function $\phi:H\rightarrow \mathbb{R}$,
\begin{equation*}
\lim_{y\rightarrow x}\sup_{t\geq0}|\mathbb{E}\phi(u_t^{x})-\mathbb{E}\phi(u_t^{y})|=0.
\end{equation*}

\cite[Theorem 2]{Kapica} implies that if a Markov process has the $e$-property and irreducibility, then the Markov process has at most one invariant measure. Therefore, combining Theorem \ref{thmmulti-Yang1} and Theorem \ref{thm2}, if either  Condition (C5') holds or
$\nu$ is symmetric, then there exists at most one invariant measure for (\ref{meq1-additive case}).

Moreover, if (\ref{5-1}) holds, then Theorem \ref{thmmulti-02} implies the existence of invariant measure for (\ref{meq1-additive case}).

The proof of Theorem \ref{thm3} is complete.
\end{proof}

In the following, we present examples of additive driving noises
satisfying (\ref{5-1}).

\begin{exm}{\bf Cylindrical L\'{e}vy process}

 Let $\{e_i,i\in \mathbb{N}\}$ be an orthonormal basis of $H$. Let $\{(L_i(t))_{t\geq0},i\in \mathbb{N}\}$ be a sequence of independent one-dimensional pure jump L\'{e}vy
processes with the intensity measure $\mu_i$. Choose $\beta_i\in \mathbb{R},i\in\mathbb{N}$ such that
       \begin{eqnarray}\label{Eq ITM 01}
       \sum_{i=1}^\infty\int_{\mathbb{R}}|\beta_ix_i|^2\wedge1\mu_i(dx_i)<\infty,
       \end{eqnarray}
       and there exists a constant $\hat\theta\in(0,1]$ such that
       \begin{eqnarray}\label{Eq ITM 02}
       \sum_{i=1}^\infty\int_{|\beta_ix_i|>1}|\beta_ix_i|^{\hat\theta}\mu_i(dx_i)<\infty.
       \end{eqnarray}
        Then, the intensity measure $\nu$ of $L(t)=\sum_{i=1}^\infty \beta_iL_i(t)e_i, t\geq0$ satisfies (\ref{5-1}). Indeed,
               \begin{eqnarray*}
       \int_{H}\|z\|^2\wedge1\nu(dz)=\sum_{i=1}^\infty\int_{\mathbb{R}}|\beta_ix_i|^2\wedge1\mu_i(dx_i)<\infty,
       \end{eqnarray*}
       and
       \begin{eqnarray*}
       \int_{Z_1}\|z\|^{\hat\theta}\nu(dz)=\sum_{i=1}^\infty\int_{|\beta_ix_i|>1}|\beta_ix_i|^{\hat\theta}\mu_i(dx_i)<\infty.
       \end{eqnarray*}

       There are many concrete examples such that $\sum_{i=1}^\infty 1_{\mathbb{R}\setminus\{0\}}(\beta_i)\mu_i(\mathbb{R})<\infty$,  and then, the driving  L\'evy process $L$ is a compound Poisson process on $H$. Another concrete example is  the so-called cylindrical  $\alpha$-stable processes with $\alpha\in(0,2)$.

\end{exm}

The following example is concerned with the subordination of L\'evy processes, which is an important idea
to obtain new L\'evy processes. Here we just introduce a concrete example, the so-called subordinated cylindrical Wiener process. We refer the reader  to \cite[Example 4.3]{WYZZ} for more details and examples.

\begin{exm}{\bf Subordinated cylindrical Wiener process}

Let $\{e_i,i\in \mathbb{N}\}$ be an orthonormal basis of $H$ and $\{W^i_t,t\geq0;i\in\mathbb{N}\}$ be a sequence of i.i.d. one-dimensional Brownian motions. Let
$\beta_i\in\mathbb{R}, i\in\mathbb{N}$ be given constants satisfying $\sum_{i=1}^\infty\beta_i^2<\infty$. Let $W_t,t\geq0$ be the
Wiener process on $H$ given by
$$
W_t=\sum_{i=1}^\infty \beta_iW^i_te_i.
$$

 For $\alpha\in(0,2)$, let
$S_t,t\geq 0$ be an $\alpha/2$-stable subordinator, i.e., an increasing one-dimensional L\'evy process with Laplace transform
$$
\mathbb{E}e^{-\eta S_t}=e^{-t|\eta|^{\alpha/2}},~~\eta>0.
$$
$W_t,t\geq0$ and $S_t,t\geq 0$ are independent. The subordinated cylindrical Wiener process $L_t, t\geq0$  on $H$ is defined by
$$
L_t:=W_{S_t}=\sum_{i=1}^\infty \beta_iW^i_{S_t}e_i.
$$
Then the intensity measure $\nu$ of $L_t, t\geq0$ is symmetric and  satisfies (\ref{5-1}).

\end{exm}

  We stress that the driving  L\'evy process $L$ introduced in the above two examples could be degenerate; that is, there could exist some $I\subset \mathbb{N}$ such that $\beta_i=0$ for any $i\in I$.

\subsection{Ergodicity: The case of multiplicative noise}
To derive the uniqueness of the invariant measure in the case of multiplicative noise, we require the following conditions:

{\bf (C8)}: Assume that $\sigma$ satisfies for all $v_1,~v_2\in H$
\begin{equation*}
  \|\sigma(v_1,z)-\sigma(v_2,z)\|\leq L(z)\|v_1-v_2\|
\end{equation*}
with the Lipschitz coefficient satisfying $\int_{H} L^2(z)\nu(dz)<\infty$ and $\int_{0<\|z\|<1}L(z)\nu(dz)<\infty$.

{\bf (C9)(Weak dissipation)}:\begin{equation*}
2\int_{Z_1}L(z)\nu(dz)+\int_{H}L^2(z)\nu(dz)+2\lambda\leq2\alpha_1\lambda_1.
\end{equation*}
\vskip 0.3cm
We now state the ergodicity of (\ref{1}) as follows:
\begin{thm}\label{thm4}
Under Conditions (C1)-(C5), (C8), and (C9), there exists at most one invariant measure for (\ref{1}), and if  that invariant measure exists, then the support of that measure is $H$. If, moreover, (C7) holds, then there exists a unique invariant measure for (\ref{1}).
\end{thm}
\begin{proof}
Applying the It\^{o} formula gives
\begin{eqnarray}\label{3-1}
&&\|u_t^{x}-u_t^{y}\|^2\nonumber\\
&=&\|x-y\|^2+2\lambda\int_{0}^{t}\|u_s^{x}-u_s^{y}\|^2ds-2\alpha_1\int_0^t\|u_s^{x}-u_s^{y}\|_V^2ds\nonumber\\
&&+2{\rm{Re}}\int_0^t\langle u_s^{x}-u_s^{y},(\alpha_2+i\beta_2)(|u_s^{x}|^{2\theta}u_s^{x}-|u_s^{y}|^{2\theta}u_s^{y})\rangle ds\nonumber\\
&&+\int_0^t\int_{Z_1}(\|u_{s-}^{x}-u_{s-}^{y}+\sigma(u_{s-}^{x},z)-\sigma(u_{s-}^{y},z)\|^2 -\|u_{s-}^{x}-u_{s-}^{y}\|^2) N(ds,dz)\nonumber\\
&&+\int_0^t\int_{Z_1^c}(\|u_{s-}^{x}-u_{s-}^{y}+\sigma(u_{s-}^{x},z)-\sigma(u_{s-}^{y},z)\|^2 -\|u_{s-}^{x}-u_{s-}^{y}\|^2) \tilde N(ds,dz)\nonumber\\
&&+\int_0^t\int_{Z_1^c}\|\sigma(u_s^{x},z)-\sigma(u_s^{y},z)\|^2\nu(dz)ds.
\end{eqnarray}
Using a standard stopping time argument, by (\ref{5-3}), Conditions (C8) and (C9), and the Poinc\'are inequality, we have
\begin{eqnarray}\label{3-0}
&&\mathbb{E}\|u_t^{x}-u_t^{y}\|^2\nonumber\\
&\leq&\|x-y\|^2+2\lambda\int_{0}^{t}\mathbb{E}\|u_s^{x}-u_s^{y}\|^2ds-2\alpha_1\int_0^t\mathbb{E}\|u_s^{x}-u_s^{y}\|_V^2ds\\
&&+\big(2\int_{Z_1}L(z)\nu(dz)+\int_{H}L^2(z)\nu(dz)\big)\mathbb{E}\int_0^t\|u_s^{x}-u_s^{y}\|^2ds\nonumber\\
&\leq&\|x-y\|^2+\big(2\int_{Z_1}L(z)\nu(dz)+\int_{H}L^2(z)\nu(dz)+2\lambda-2\alpha_1\lambda_1\big)\int_0^t\mathbb{E}\|u_s^{x}-u_s^{y}\|^2ds.\nonumber
\end{eqnarray}
If $2\int_{Z_1}L(z)\nu(dz)+\int_{H}L^2(z)\nu(dz)+2\lambda<2\alpha_1\lambda_1$, then by (\ref{3-0}) there exists a constant $c>0$ such that
\begin{equation*}
\mathbb{E}\|u_t^{x}-u_t^{y}\|^2\leq\|x-y\|^2e^{-ct},
\end{equation*}
which implies the uniqueness of the invariant measure. Theorem \ref{thmmulti-Yang} implies that if invariant measures exist, then the support of any invariant measure is $H$. If $2\int_{Z_1}L(z)\nu(dz)+\int_{H}L^2(z)\nu(dz)+2\lambda=2\alpha_1\lambda_1$, then (\ref{3-0}) implies that  $\{u^x;x\in H\}$  satisfies the $e$-property. Combining  Theorem \ref{thmmulti-Yang} and \cite[Theorem 2]{Kapica}, there exists at most one invariant measure for (\ref{1}), and if invariant measures exist, then the support of any invariant measure is $H$.
We remark that the novelty of our result in this subsection
concerns the case $2\int_{Z_1}L(z)\nu(dz)+\int_{H}L^2(z)\nu(dz)+2\lambda=2\alpha_1\lambda_1$.

If, moreover, (C7) holds, then Theorem \ref{thmmulti-02} implies the existence of invariant measure for (\ref{1}).

The proof of Theorem \ref{thm4} is complete.
\end{proof}
\vskip 0.4cm
\noindent{\bf Acknowledgement}. This work is partially  supported by NSFC (No. 12131019, 11971456, 11721101).

\end{document}